\documentclass[12pt]{amsart}

\usepackage{amsmath, amssymb, bbm}
\DeclareMathOperator{\trace}{trace}
\usepackage{enumitem}
\newtheorem{theorem}{Theorem}[section]
\newtheorem{lemma}[theorem]{Lemma}
\newtheorem{prop}[theorem]{Proposition}

\theoremstyle{definition}

\theoremstyle{remark}
\newtheorem{remark}[theorem]{Remark}

\numberwithin{equation}{section}
\usepackage{appendix}

\newcommand{\rr}{{\mathbb R}}
\newcommand{\rd}{{\mathbb R^d}}
\newcommand{\rmm}{{\mathbb{R}^m}}

\newcommand{\Ker}{\operatorname{Ker}}

\newcommand{\Gr}{\operatorname{Gr}}
\newcommand{\dimh}{\dim_{\mathcal{H}}}
\newcommand{\cov}{\operatorname{Cov}}

\allowdisplaybreaks

\begin{document}
\sloppy
\title[Hausdorff dimension of operator-self-similar Gaussian random fields]{The Hausdorff dimension of\\ multivariate operator-self-similar Gaussian random fields} 

\author{Ercan S\"onmez}
\address{Ercan S\"onmez, Mathematisches Institut, Heinrich-Heine-Universit\"at D\"usseldorf, Universit\"atsstr. 1, D-40225 D\"usseldorf, Germany}
\email{ercan.soenmez\@@{}hhu.de} 

\date{\today}

\begin{abstract}
Let $\{X(t) : t \in \rd \}$ be a multivariate operator-self-similar random field with values in $\rmm$. Such fields were introduced in \cite{LiXiao} and satisfy the scaling property $\{X(c^E t) : t \in \rd \} \stackrel{\rm d}{=} \{c^D X(t) : t \in \rd \}$ for all $c>0$, where $E$ is a $d \times d$ real matrix and $D$ is an $m \times m$ real matrix. We solve an open problem in \cite{LiXiao} by calculating the Hausdorff dimension of the range and graph of a trajectory over the unit cube $K = [0,1]^d$ in the Gaussian case. In particular, we enlighten the property that the Hausdorff dimension is determined by the real parts of the eigenvalues of $E$ and $D$ as well as the multiplicity of the eigenvalues of $E$ and $D$.
\end{abstract}

\maketitle

\baselineskip=18pt

\section{Introduction}

In this paper we consider multivariate operator-self-similar random fields as introduced in \cite{LiXiao}. More precisely, let $E \in \mathbb{R}^{d \times d}$ and $D \in \mathbb{R}^{m \times m}$ be real matrices with positive real parts of their eigenvalues. A  random field $\{X(t) : t \in \rd \}$ with values in $\rmm$ is called multivariate operator-self-similar for $E$ and $D$ or $(E,D)$-operator-self-similar if

\begin{equation}\label{OSS}
	\{ X( c^Ex) : x \in \mathbb{R}^d \} \stackrel{\rm d}{=} \{ c^D X(x) : x \in \mathbb{R}^d \}  \quad \text{for all } c>0,
\end{equation}
where $\stackrel{\rm d}{=}$ means equality of all finite-dimensional  marginal distributions and, as usual, $c^A = \exp (A \log c) = \sum_{k=0}^{\infty} \frac{(\log c)^k}{k!} A^k$ is the matrix exponential for every matrix $A$. In the literature $\rd$ is usually referred to as the time-domain, $\rmm$ as the state space and $X$ is called a $(d,m)$-random field. Furthermore, $E$ is called the time-variable scaling exponent and $D$ the state space scaling operator.

These fields can be seen as a generalization of both operator-self-similar processes (see \cite{LahaRo, HudsonMason, Sato}) and operator scaling random fields (see \cite{BMS, BL}). Let us recall that a stochastic process $\{X(t) : t \in \rr \}$ with values in $\rmm$ is called operator-self-similar if
\begin{equation*}
	\{ X( c t) : t \in \mathbb{R} \} \stackrel{\rm d}{=} \{ c^D X(t) : t \in \mathbb{R} \}  \quad \text{for all } c>0,
\end{equation*}
whereas a scalar valued random field $\{ Y( t) : t \in \mathbb{R}^d \}$ is said to be operator scaling for $E$ and some $H>0$ if
\begin{equation*}
	\{ Y( c^E t) : t \in \mathbb{R}^d \} \stackrel{\rm d}{=} \{ c^H Y(t) : t \in \mathbb{R}^d \}  \quad \text{for all } c>0.
\end{equation*}
Operator-self-similar processes have been studied extensively during the past decades due to their theoretical importance and are also used in various applications such as physics, engineering, biology, mathematical finance, just to mention a few (see e.g. \cite{Levy, Abry, SamorodTaqq, Wackernagel, Chiles}). Several authors proposed to apply operator scaling random fields for modeling phenomena in spatial statistics, hydrology and image processing (see e.g. \cite{BonEstr, Benson, Davis}).

A very important class of operator-self-similar random fields are given by Gaussian random fields and especially by the so-called operator-fractional Brownian motion $B_D$ with state space scaling operator $D$ (see \cite{MasonXiao}), also known as operator-fractional Brownian field \cite{Baek, Didier}. The random field $B_D$ fulfills the self-similarity relation
\begin{equation*}
	\{ B_D( c t) : t \in \mathbb{R}^d \} \stackrel{\rm d}{=} \{ c^D B_D(t) : t \in \mathbb{R}^d \} 
\end{equation*}
and has stationary increments, i.e. it satisfies
\begin{equation*}
	\{ B_D( t+h) - B_D(h) : t \in \mathbb{R}^d \} \stackrel{\rm d}{=} \{ B_D(t) : t \in \mathbb{R}^d \} 
\end{equation*}
for any $h \in \rd$. We remark that Mason and Xiao \cite{MasonXiao} studied several sample path properties of $B_D$ including fractal dimensions of the range and the graph of $B_D$. More precisely, for any arbitrary Borel set $F \subset \rd$, under some additional assumptions (see \cite[Theorem 4.1]{MasonXiao}), they showed that a.s. the Hausdorff dimension of the range and graph are given by
\begin{align}\label{dimMason}
	\dim_{\mathcal{H}} B_D (F) &= \min \Big\{ m, \Big( \dimh F + \sum_{i=1}^j (\lambda_j - \lambda_i) \Big) 						\lambda_j^{-1}, j=1, \ldots, m \Big\} , \\
	\dim_{\mathcal{H}} \Gr B_D (F) &= \begin{cases} 
							\begin{aligned}
							& \dimh B_D(F) & \text{ if } \dimh F \leq \sum_{i=1}^m \lambda_i ,\\ 									& \dimh F + \sum_{i=1}^m (1- \lambda_i)  & \text{ if } \dimh F > 												\sum_{i=1}^m \lambda_i,
							\end{aligned}
						\end{cases}
\end{align}
where $0< \lambda_1 \leq \lambda_2 \leq \ldots \leq \lambda_m<1$ are the real parts of the eigenvalues of $D$. In particular, if $F = [0,1]^d$ they obtain that a.s.
\begin{align}\label{specialdim}
	\dim_{\mathcal{H}} B_D (F) &= \min \Big\{ m, \Big( d + \sum_{i=1}^j (\lambda_j - \lambda_i) \Big) 						\lambda_j^{-1}, j=1, \ldots, m \Big\} , \\
	\dim_{\mathcal{H}} \Gr B_D (F) &= \begin{cases} 
							\begin{aligned}
							& \dimh B_D(F) & \text{ if } d \leq \sum_{i=1}^m \lambda_i ,\\ 									& d + \sum_{i=1}^m (1- \lambda_i)  & \text{ if } d > 												\sum_{i=1}^m \lambda_i,
							\end{aligned}
						\end{cases}
\end{align}

The random field $B_D$ is isotropic in the sense that for every $t \in \rd$ we have
$$B_D(t) \stackrel{\rm d}{=} \| t \|^D B_D (1,0, \ldots, 0),$$
see \cite[Theorem 3.1]{MasonXiao}. Moreover, it is a generalization of the famous $m$-dimensional fractional Brownian motion, implicitely introduced in \cite{Kolm} and defined in \cite{Mandelbrot} in dimension $m=1$. However, certain applications (see e.g. \cite{BonEstr, Benson})  require that the random field is anisotropic. To reach this goal, Bierm\'e, Meerschaert and Scheffler \cite{BMS} introduced operator scaling random fields and provided the existence of two different classes of such $\alpha$-stable ($0<\alpha \leq 2$) fields given by harmonizable as well as moving average stochastic integral representations. In the Gaussian case $\alpha=2$ they showed that the moving average and harmonizable fields have the same kind of regularity properties including results about H\"older continuity and fractal dimensions. Bierm\'e and Lacaux \cite{BL} showed that this is no more true in the stable case $\alpha \in (0,2)$.

By defining stochastic integral representations for random vectors and following the outline in \cite{BMS}, Li and Xiao \cite{LiXiao} established the existence of multivariate operator-self-similar $\alpha$-stable random fields satisfying the scaling relation \eqref{OSS}. Furthermore, they mention that from both theoretical and applied point of view it would be interesting to study their sample path regularity and fractal poperties. The purpose of this paper is to provide the related results in the Gaussian case $\alpha=2$ and to generalize several results in the literature including \eqref{specialdim} and (1.5).

A main tool for the study of sample paths of multivariate $(E,D)$-operator-self-similar Gaussian random fields is the change to polar coordinates with respect to the matrix $E$ introduced in \cite{MeersScheff} and used in \cite{BMS, BL}. If $X$ is a multivariate $(E,D)$-operator-self-similar Gaussian random field with stationary increments and $X(0) = 0$ almost surely, using \eqref{OSS} we can write the covariance matrix of $X$ as
$$ \operatorname{Cov} [ X(t) ] = \tau_E (t)^D \cov [ X\big( l_E(t) \big) ] \tau_E(t)^{D^*}$$
for all $t \in \rd$, where $D^*$ is the adjoint of $D$, $\tau_E(t)$ is the radial part of $t$ with respect to $E$ and $l_E(t)$ is its directional part, as introduced in \cite{BMS}. We will recall the precise definition of these generalized polar coordinates in Section 3. Therefore, many sample path properties depend on the polar coordinates $\big( \tau_E(t), l_E(t) \big)$ and the real parts of the eigenvalues of the linear operator $D$. The radial part $\tau_E(t)$ can be considered as a function. In particular, Lemma 2.2 in \cite{BMS} shows that $\tau_E(t-s)$ can be regarded as a quasi-metric on $\rd$ (see e.g. \cite{Stempak} for a definition) and it has been used extensively to study operator scaling random fields (see \cite{BMS, BL, BL2, LiWangXiao}). 

Furthermore, to investigate the sample paths of multivariate $(E,D)$-operator-self-similar random fields another main tool we use is the spectral decomposition of $\rd$ with respect to $E$ introduced in \cite[Section 2.1]{MeersScheff}. Actually, it turns out to be necessary in order to formulate our results if the space dimension $m$ satisfies $m \geq 2$.

In Section 2 we recall the definition of moving average and harmonizable multivariate $(E,D)$-operator-self-similar Gaussian random fields. Section 3 is devoted to the main tools we need for the study of these fields. More precisely, we recall the definition and some well-known results about the polar coordinates and the spectral decomposition with respect to $E$. Based on this tools we furthermore present and prove a general Lemma which will be needed in order to determine an upper bound for the Hausdorff dimension of the range and the graph. Finally, in Section 4 we state and prove our main results on multivariate $(E,D)$-operator-self-similar Gaussian random fields, including sample path continuity and Hausdorff dimensions. It will be clear that the methods used in \cite{MasonXiao, BMS, AychXiao} play important roles in this paper.

\section{Moving average and harmonizable representation}

Let us recall the definition of moving average and harmonizable multivariate operator-self-similar Gaussian random fields given in \cite{LiXiao}. Throughout this paper, let $E \in \mathbb{R}^{d \times d}$ be a matrix with distinct positive real parts of its eigenvalues given by $0 < a_1 < \ldots < a_p$ for some $1 \leq p \leq d$ and let $D \in \mathbb{R}^{m \times m}$ be a matrix with positive real parts of its eigenvalues given by $0 < \lambda_1 \leq \lambda_2 \leq \ldots \leq \lambda_m$. Note that $\lambda_1, \ldots , \lambda_m$ are not necessarily different.

In the following let $\phi : \rd \to [0, \infty )$ be an $E$-homogeneous $(\beta , E)$-admissible function according to Definition 2.6 and Definition 2.7 in \cite{BMS}. Recall that a function $f: \rd \to \mathbb{C}$ is called $E$-homogeneous if $f (c^{E} x) = c f(x)$ for all $c>0$ and $x \in \rd \setminus \{0\}$, whereas a function $g: \rd \to [0, \infty) $ is called $(\beta, E)$-admissible with $\beta >0$ if $g(x) > 0$ for all $x \neq 0$ and for any $0 < A < B$ there exists a positive constant $C > 0$ such that for $A \leq \|y\| \leq B$
$$ \tau_E (x) \leq 1 \Rightarrow |g(x+y)-g(y) | \leq  \tau_E (x)^\beta,$$
where $\tau_E(x)$ is the radius of $x$ with respect to $E$ introduced in Section 3 below. Various examples of such functions have been constructed in \cite{BMS, BL}. Since $\phi$ is $(\beta , E)$-admissible, Remark 2.9 in \cite{BMS} implies that $0 < \beta \leq a_1$.

Let $M_2 (d \xi)$ be an $\rmm$-valued symmetric Gaussian random measure on $\rd$ with Lebesgue control measure. Such a measure has been defined in \cite[Definition 2.1]{LiXiao} and is the generalization of the random measure introduced in \cite[p. 121]{SamorodTaqq} to higher space dimensions. Furthermore, let $q= \operatorname{trace} (E)$ and $I_m$ the identity operator on $\rmm$.

\begin{theorem}\label{movingaverage} \cite[Theorem 2.5]{LiXiao}
If $\lambda_m < \beta$, then the random field
\begin{equation}\label{mov}
X_{\phi} (x) = \int_{\rd} [ \phi(x-y)^{D-\frac{qI_m}{2}} - \phi(-y)^{D-\frac{qI_m}{2}} ] M_{2} (d y) , \quad x \in \rd
\end{equation}
is well defined and called moving average $(E,D)$-operator-self-similar random field.
\end{theorem}

For the sake of simplicity let us denote the kernel matrix by
$$Q(x,y) = [ \phi(x-y)^{D-\frac{qI_m}{2}} - \phi(-y)^{D-\frac{qI_m}{2}} ]$$
and let us recall that according to \cite[Theorem 2.3]{LiXiao} $X_\phi$ exists, since
$$\int_{\rd} \| Q(x,y) \| _m^2 dy < \infty$$
for all $x \in \rd$, where $\| Q \| _m = \max_{\|u\|=1} \| Qu \|$ is the operator norm for all matrices $Q \in \mathbb{R}^{m \times m}$. See \cite[Section 4]{LiXiao} for details on the construction.

Li and Xiao \cite[Theorem 2.4]{LiXiao} also define stochastic integrals of complex matrix-valued functions with respect to a $\mathbb{C}^m$-valued symmetric Gaussian random measure, denoted $W_2 (dy)$, which is more general than the stochastic integrals defined in \cite[p. 281]{SamorodTaqq}. Let $E^*$ be the adjoint of $E$ and suppose now that $\psi : \rd \to [0, \infty)$ is a continuous $E^*$-homogeneous function such that $\psi(x) \neq 0$ for $x \neq 0$.

\begin{theorem}\label{harmonizable} \cite[Theorem 2.6]{LiXiao}
If $\lambda_m < a_1$ the random field
\begin{equation}\label{harmon}
X_{\psi} (x) = \operatorname{Re} \int_{\rd} ( e^{i \langle x, y \rangle} -1 ) \psi (y) ^{-D - \frac{q I_m}{2}} W_{2} (d y) , \quad x \in \rd
\end{equation}
is well defined and called harmonizable $(E,D)$-operator-self-similar random field.
\end{theorem}
As in the above $X_\psi$ is well defined since the kernel matrix in \eqref{harmon} satisfies
$$\int_{\rd} ( |1-\cos \langle x,y \rangle |^2 + |\sin \langle x,y \rangle |^2 ) \| \psi (y) ^{-D - \frac{q I_m}{2}} \| _m^2 dy < \infty$$
for all $x \in \rd$.

From Theorem 2.5 and Theorem 2.6 in \cite{LiXiao} the random fields given in \eqref{mov} and \eqref{harmon} are stochastically continuous, have stationary increments and satisfy the scaling property \eqref{OSS}. In addition, it is shown that $X_\psi$ is proper, whereas $X_\phi$ is proper if $\frac{q}{2}$ is not an eigenvalue of $D$ (see \cite[Remark 2.1]{LiXiao}). Let us recall that an $\rmm$-valued random field $\{ Y(t) : t \in \rd\}$ is said to be proper if for every $t \in \rd$ the distribution of $Y(t)$ is full, i.e. it is not contained in any proper hyperplane in $\rmm$. In the Gaussian case it is well known that the latter is equivalent to $\det \cov [ Y(t)] >0$. For the sake of simplicity we will always assume that $\frac{q}{2}$ is not an eigenvalue of $D$, i.e. that $X_\phi$ is proper as well. Furthermore, under the assumptions of Theorem \ref{movingaverage} and Theorem \ref{harmonizable} up to considering the matrix $\frac{E}{H}$ for some $H \in (\lambda_m, a_1)$ without loss of generality we will assume that
\begin{equation}\label{realparts}
0<\lambda_1 \leq \lambda_2 \leq \ldots \leq \lambda_m <1<a_1 < \ldots < a_p .
\end{equation}

The main tool in the study of the fields given in \eqref{mov} and \eqref{harmon} is the change of coordinates to the polar coordinates with respect to the matrix $E$. Therefore, before studying the sample paths of $X_\phi$ and $X_\psi$, we recall in the next Section the definition of these polar coordinates, give some estimates on the radial part and recall a spectral decomposition result from \cite{MeersScheff}. This will be needed in order to introduce a lemma in Section 4 given in terms of the radial part and the spectral decomposition from which we will obtain an a.s. upper bound for the Hausdorff dimension of the range and the graph of $X_\phi$ and $X_\psi$ over $[0,1]^d$.

\section{Polar coordinates and spectral decomposition}

Following \cite[Section 2.1]{MeersScheff} we now construct the spectral decomposition of $\rd$ with respect to $E$. See \cite[Section 3]{BL} for a different way of construction. Since the real parts of the eigenvalues of $E$ satisfy $1<a_1 < \ldots < a_p$ for some $p \leq d$ we can factor the minimal polynomial of $E$ into $f_1, \ldots, f_p$, where all roots of $f_i$ have real part equal to $a_i$. We further define $W_{i}=\Ker\big( f_{i}(E) \big)$. Then, by \cite[Theorem 2.1.14]{MeersScheff},
$$\rd=W_{1}\oplus\ldots\oplus W_{p}$$
is a direct sum decomposition, i.e. we can write any $x \in \mathbb{R}^d$ uniquely as
$$ x = x_1 + \ldots + x_p$$
for $x_i \in W_i$, $1 \leq i \leq p$. Further, we can choose an inner product on $\rd$ such that the subspaces $W_1, \ldots, W_p$ are mutually orthogonal and, throughout this paper, for any $x \in \rd$ we will choose $ \| x \| = \langle x,x\rangle^{1/2}$ as the Euclidean norm.

We now recall results about the change to polar coordinates which have already been used widely in the study of operator scaling random fields (see \cite{BMS, BL, LiWangXiao, BL2, LiXiao}). According to \cite[Section 2]{BMS} there exists a norm $\| \cdot \|_E$ on $\rd$ such that for the unit sphere $S_E = \{ x \in \rd : \|x \|_E = 1 \}$ the mapping $\Psi_E : (0, \infty) \times S_E \to \rd  \setminus \{0 \}$ defined by $\Psi_E (r, \theta) = r^E \theta$ is a homeomorphism. Thus, we can write any $x \in \rd  \setminus \{0 \}$ uniquely as
$$ x = \tau_E (x)^E l_E (x),$$
where $\tau_E (x) >0$ is called the radius of $x$ with respect to $E$ and $l_E (x) \in S_E= \{ x \in \rd : \tau_E (x) = 1 \}$ is called the direction. It is clear that $\tau_E (x) \rightarrow \infty$ as $\| x\| \rightarrow \infty$ and $\tau_E (x) \rightarrow 0$ as $\| x\| \rightarrow 0$. Further, $\tau_E( \cdot )$ is continuous on $\rd  \setminus \{0 \}$ and can be extended continuously to the whole space by setting $\tau_E(0) = 0$.

The following result gives bounds on the growth rate of $\tau_E( x )$ in terms of $a_1, \ldots, a_p$ and can easily be deduced from \cite[Lemma 2.1]{LiWangXiao} and \cite[Lemma 2.2]{LiWangXiao}. We refer to \cite{BMS,BL} for more refined results on the bounds.

\begin{lemma}\label{pcbounds}
Let $H>0$ and $\rd = W_1 \oplus \ldots \oplus W_p$ the direct sum decomposition with respect to $E$. For any $1 \leq k \leq p$ define $O_k = W_1 \oplus \ldots \oplus W_k$. For any $\varepsilon >0$ small enough and any $\eta \in (0,1)$ there exist constants $C_{1,1}, C_{1,2} >0$ such that
$$ C_{1,1} \Big( \| y \|^{\frac{H+\varepsilon}{a_1}} + \sum_{i=k+1}^p \| x_i \| ^{\frac{H+\varepsilon}{a_i}} \Big) \leq \tau_E(x)^H \leq   C_{1,2} \Big( \| y\|^{\frac{H-\varepsilon}{a_k}} + \sum_{i=k+1}^p \| x_i \| ^{\frac{H-\varepsilon}{a_i}} \Big),$$
where $x = y + \sum_{i=k+1}^p x_i$ with $ \| x\| \leq \eta$, $y\in O_k$ and $x_i \in W_i$, $k+1 \leq i \leq p$.
\end{lemma}
Using the spectral decomposition and the estimates given in Lemma \ref{pcbounds} we state and prove our results on the modulus of continuity of $X_\phi$ and $X_\psi$ and determine the Hausdorff dimension of their sample paths in the next section.

\section{Modulus of continuity and Hausdorff dimension}

Throughout this section let $X= \{X(t) : t \in \rd \}$ be a proper, Gaussian and $(E,D)$-operator-self-similar random field in $\rmm$ with stationary increments. Examples of such random fields are given by the moving average and harmonizable random fields introduced in Section 2. In order to formulate our results conveniently let us define $\tilde{a}_j = a_{p+1-j}$ and $\tilde{W}_j = W_{p+1-j}$ for $j = 1, \ldots, p,$ where $W_1, \ldots, W_p$  is the spectral decomposition according to Section 3. Note that
$$\tilde{a}_1 > \ldots > \tilde{a}_p .$$
We now state our main results and refer the reader to \cite{Fal, Mattila} for the definition and properties of the Hausdorff dimension.

\begin{theorem}\label{hausdorff}
Assume that \eqref{realparts} holds. Then with probabilty one
\begin{align}\label{imdim}
	\dim_{\mathcal{H}} X ([0,1]^d) &= \min \Big\{ m, \frac{ \sum_{k=1}^p a_k \dim W_k + \sum_{i=1}^j (\lambda_j - \lambda_i) }{ \lambda_j }, 1 \leq j \leq m \Big\}
\end{align}

\begin{align}\label{imdim2}
	 &= \begin{cases}		
				 m & \mbox{if } \sum_{i=1}^m \lambda_i < \sum_{k=1}^p a_k \dim W_k ,\\ 				 \frac{ \sum_{k=1}^p a_k \dim W_k + \sum_{i=1}^l (\lambda_l - \lambda_i) }{ \lambda_l }  & \mbox{if }  \sum_{i=1}^{l-1} \lambda_i < \sum_{k=1}^p a_k \dim W_k \leq \sum_{i=1}^{l} \lambda_i ,		
		\end{cases}
\end{align}

\begin{align}\label{graphdim}
	\dim_{\mathcal{H}} \Gr X ([0,1]^d) &= \min \Bigg\{ \frac{ \sum_{k=1}^p a_k \dim W_k + \sum_{i=1}^j (\lambda_j - \lambda_i) }{ \lambda_j }, 1 \leq j \leq m, \\ 
	& \quad \sum_{j=1}^l \frac{\tilde{a}_j}{\tilde{a}_l} \dim \tilde{W}_j + \sum_{j=l+1}^p \dim \tilde{W}_j + \sum_{i=1}^m (1- \frac{\lambda_i}{\tilde{a}_l}) , 1 \leq l \leq p \Bigg\} \nonumber
\end{align}
\begin{footnotesize}
\begin{align}\label{graphdim2}
	 &= \begin{cases} 
			\begin{aligned}
				& \dim_{\mathcal{H}} X ([0,1]^d) \qquad\qquad\qquad\qquad\qquad\qquad\,\,\, \mbox{ if } \sum_{k=1}^p a_k \dim W_k \leq \sum_{i=1}^m \lambda_i  ,\\ 		
		&   \sum_{j=1}^l \frac{\tilde{a}_j}{\tilde{a}_l} \dim \tilde{W}_j + \sum_{j=l+1}^p \dim \tilde{W}_j + \sum_{i=1}^m (1- \frac{\lambda_i}{\tilde{a}_l}) \mbox{ if } \sum_{k=1}^{l-1} \tilde{a}_k \dim \tilde{W}_k \leq \sum_{i=1}^m \lambda_i < \sum_{k=1}^{l} \tilde{a}_k \dim \tilde{W}_k .
			\end{aligned}
		\end{cases}
\end{align}
\end{footnotesize}
\end{theorem}
\begin{remark}
In \eqref{realparts} we made use of the fact that the matrices $E$ and $D$ are in general not unique. However, Theorem \ref{hausdorff} shows that the quotient of the real parts of the eigenvalues of $E$ and $D$ is always unique. Moreover note that $\sum_{k=1}^p a_k \dim W_k = q = \trace (E)$ so that the Hausdorff dimension of the range only depends on the trace of $E$.
\end{remark}
We remark that in the special case $E=I_d$, where $I_d$ is the identity operator in $\rd$, Theorem \ref{hausdorff} coincides with \eqref{specialdim} and (1.5). Furthermore, Theorem \ref{hausdorff} generalizes several other results, such as \cite[Theorem 5.6]{BMS} and the Hausdorff dimension results stated in \cite[Section 3]{LiWangXiao} about $(d,m)$-random fields where the components are independent copies of the operator scaling Gaussian random fields constructed in \cite{BMS}.

Equalities \eqref{imdim2} and \eqref{graphdim2} are verified by Lemma \ref{cases} in the Appendix whose proof is elementary and omitted.

The following lemma is a generalization of \cite[Lemma 2.1]{Xiao3} and will be needed in order to establish the upper bounds in Theorem \ref{hausdorff}. Furthermore, it enlightens the fact that upper bounds of the Hausdorff dimension of the range and graph of any function satisfying a generalized H\"older continuity with respect to the radius $\tau_E ( \cdot )$ are related to the spectrum of $E$, since an efficient covering of the aforementioned sets can be constructed in terms of the spectral decomposition with respect to $E$.

\begin{lemma}\label{upperbound}
Let $f = (f_1, \ldots, f_m) : [0,1]^d \to \rmm$ satisfy
\begin{align}\label{fbound}
|f_i(x) - f_i(y)| \leq c \, \tau_E (x-y)^{\alpha_i} , \quad i=1, \ldots, m,
\end{align}
where $c>0$ and
$$0<\alpha_1 \leq \alpha_2 \leq \ldots \leq \alpha_m \leq 1.$$
Then
\begin{align}
\dim_{\mathcal{H}} f ([0,1]^d) &\leq \min \Big\{ m, \frac{ \sum_{k=1}^p a_k \dim W_k + \sum_{i=1}^j (\alpha_j - \alpha_i) }{ \alpha_j }, 1 \leq j \leq m \Big\}, \\
\dim_{\mathcal{H}} \Gr f ([0,1]^d) &\leq \min \Bigg \{\frac{ \sum_{k=1}^p a_k \dim W_k + \sum_{i=1}^j (\alpha_j - \alpha_i) }{ \alpha_j }, 1 \leq j \leq m, \\ 
	& \quad \sum_{j=1}^l \frac{\tilde{a}_j}{\tilde{a}_l} \dim \tilde{W}_j + \sum_{j=l+1}^p \dim \tilde{W}_j + \sum_{i=1}^m (1- \frac{\alpha_i}{\tilde{a}_l}) , 1 \leq l \leq p \Bigg\}. \nonumber
\end{align}
\end{lemma}

\begin{proof}
Throughout this proof, let $c$ be an unspecified constant which might change in every occurence. We clearly have $\dim_{\mathcal{H}} f ([0,1]^d) \leq m$ and $\dim_{\mathcal{H}} f ([0,1]^d) \leq \dim_{\mathcal{H}} \Gr f ([0,1]^d)$. So it suffices to prove (4.7). Let $\rd=W_{1}\oplus\ldots\oplus W_{p}$ be the direct sum decomposition with respect to $E$ according to Section 3. We first show that
\begin{align}\label{lgbound}
\dim_{\mathcal{H}} \Gr f ([0,1]^d) &\leq \frac{ \sum_{k=1}^p a_k \dim W_k + \sum_{i=1}^j (\alpha_j - \alpha_i) }{ \alpha_j }
\end{align}
for every $1 \leq j \leq m$. We can choose compact subsets $V_1 \subset W_1, \ldots, V_p \subset W_p$ such that
$$[0,1]^d \subset V_1 + \ldots + V_p,$$
where $V_1 + \ldots + V_p = \{ x_1 + \ldots + x_p : x_i \in V_i , 1 \leq i \leq p \}$. For any integer $n \geq 2$, we divide $V_l$ $(1 \leq l \leq p)$ into $k_{n,l}$ cubes $\{ R_{n,l,i_l} \}$ $(1 \leq i_l \leq k_{n,l} )$ with edge-lengths $n^{-a_l}$. Then
\begin{align}\label{numberrectangles}
k_{n,1} \cdot \ldots \cdot k_{n,p} \leq c \, n^{\sum_{l=1}^p a_l \dim W_l}.
\end{align}
By \eqref{fbound} and Lemma \ref{pcbounds}, for any small $\varepsilon >0$, each $f(R_{n,1,i_1} + \ldots + R_{n,p,i_p})$ can be covered by a rectangle $T_{n,i_1, \ldots, i_p} \subset \rmm $ of sides with respective length $c (\frac{1}{n})^{\alpha_{i}-\varepsilon}$ for $1 \leq i \leq m$. For each fixed $1 \leq j \leq m$, $T_{n,i_1, \ldots, i_p}$ can be covered by $c (\frac{1}{n})^{\sum_{i=1}^j (\alpha_i - \alpha_j ) -\varepsilon}$ cubes $T_{n,i_1, \ldots, i_p,k}$ of edge-lengths $c (\frac{1}{n})^{\alpha_j}$. Note that
$$\dim_{\mathcal{H}} \Gr f ([0,1]^d) \subset \bigcup_{i_1, \ldots , i_p} \bigcup_{k} (R_{n,1,i_1} + \ldots + R_{n,p,i_p}) \times T_{n,i_1, \ldots, i_p,k} $$
and
\begin{align}\label{lgdiam}
\operatorname{diam} \big( (R_{n,1,i_1} + \ldots + R_{n,p,i_p}) \times T_{n,i_1, \ldots, i_p,k} \big) \leq c (\frac{1}{n})^{\alpha_j}.
\end{align}
Let $\gamma > \varepsilon$. Then, by \eqref{numberrectangles} and \eqref{lgdiam}, we have
\begin{align*}
& \sum_{i_1, \ldots, i_p} \sum_k \operatorname{diam} \big( (R_{n,1,i_1} + \ldots + R_{n,p,i_p}) \times T_{n,i_1, \ldots, i_p,k} \big)^{\big[ \gamma + \sum_{l=1}^p a_l \dim W_l + \sum_{i=1}^j (\alpha_j - \alpha_i)\big]/\alpha_j} \\
& \leq c \,k_{n,1} \cdot \ldots \cdot k_{n,p} \cdot (\frac{1}{n})^{\sum_{i=1}^j (\alpha_i - \alpha_j ) -\varepsilon} \cdot (\frac{1}{n})^{\gamma + \sum_{l=1}^p a_l \dim W_l + \sum_{i=1}^j (\alpha_j - \alpha_i)} \\
& \leq c \, (\frac{1}{n})^{\gamma-\varepsilon} \to 0
\end{align*}
as $n \to \infty$. This proves
\begin{align*}
\dim_{\mathcal{H}} \Gr f ([0,1]^d) &\leq \frac{\varepsilon + \sum_{k=1}^p a_k \dim W_k + \sum_{i=1}^j (\alpha_j - \alpha_i) }{ \alpha_j }
\end{align*}
for all $1 \leq j \leq m$ and $\varepsilon>0$. Hence, \eqref{lgbound} follows by letting $\varepsilon \to 0$. It remains to prove
\begin{align}\label{lgbound2}
\dim_{\mathcal{H}} \Gr f ([0,1]^d) &\leq \sum_{j=1}^k \frac{\tilde{a}_j}{\tilde{a}_k} \dim \tilde{W}_j + \sum_{j=k+1}^p \dim \tilde{W}_j + \sum_{i=1}^m (1- \frac{\alpha_i}{\tilde{a}_k})
\end{align}
for any $1 \leq k \leq p$. We fix an integer $1 \leq k \leq p$. Observe that each $$(R_{n,1,i_1} + \ldots + R_{n,p,i_p}) \times T_{n,i_1, \ldots, i_p}$$
can be covered by $\ell_{n,k}$ cubes in $\mathbb{R}^{d+m}$ of side-lengths $n^{-\tilde{a}_k}$. Further, note that, since the side-lengths of $T_{n,i_1, \ldots, i_p}$ are at most $c (\frac{1}{n})^{\alpha_{i}-\varepsilon}$, $1 \leq i \leq m$, by the definition of $\tilde{a}_1, \ldots, \tilde{a}_p$, we have
\begin{align*}
\ell_{n,k} \leq c \, n^{\sum_{j=k+1}^p (\tilde{a}_k-\tilde{a}_j) \dim \tilde{W}_j - \sum_{i=1}^m (\alpha_i-\varepsilon-\tilde{a}_k)}.
\end{align*}
Hence, $\Gr f ([0,1]^d)$ can also be covered by $k_{n,1} \ldots k_{n,p} \cdot \ell_{n,k}$ cubes in $\mathbb{R}^{d+m}$ with edge-lengths $n^{-\tilde{a}_k}$. We now choose $0<\alpha_i'<\alpha_i-\varepsilon$, $1 \leq i \leq m$, and denote
$$\eta_k = \sum_{j=1}^k \frac{\tilde{a}_j}{\tilde{a}_k} \dim \tilde{W}_j + \sum_{j=k+1}^p \dim \tilde{W}_j + \sum_{i=1}^m \big( 1- \frac{\alpha_i'}{\tilde{a}_k}\big) .$$
Some simple calculations show that
\begin{align*}
& k_{n,1}  \ldots k_{n,p} \cdot \ell_{n,k} \cdot (n^{-\tilde{a}_k})^{\eta_k} \\
& \leq n^{\sum_{i=1}^m \big( \alpha_i' - (\alpha_i-\varepsilon) \big)} \to 0
\end{align*}
as $n \to \infty$, since $0<\alpha_i'<\alpha_i-\varepsilon$. This implies that
$$ \dim_{\mathcal{H}} \Gr f ([0,1]^d) \leq \eta_k.$$
Therefore, \eqref{lgbound2} follows by letting $\alpha_i' \to \alpha_i - \varepsilon$ and $\varepsilon \to 0$. The proof of Lemma \ref{upperbound} is complete. 
\end{proof}

\begin{remark} \label{remark1}
In view of the proof of Lemma \ref{upperbound} it is easy to see that if $f$ satisfies \eqref{fbound} with $\alpha_i$ replaced by $\beta_i$ for every $\beta_i < \alpha_i$, then  (4.6) and (4.7) are still valid.
\end{remark}

Recall that from the Jordan decomposition theorem (see e.g. \cite[p. 129]{Hirsch}) there exists a real invertible matrix $A \in \mathbb{R}^{m \times m}$ such that $A^{-1} D A$ is of the real canonical form, i.e. it consists of diagonal blocks which are either Jordan cell matrices of the form
\begin{equation*}
\left(
\begin{array}{cccccc}
\lambda & 1 &  & &    \\
 & \lambda & 1 &   \\
 &  &  \ddots & \ddots&   \\
 & & & \ddots & 1   \\
 & & & &  \lambda   
\end{array} \right) .
\end{equation*}
with $\lambda$ a real eigenvalue of $D$ or blocks of the form
\begin{equation*}
\left(
\begin{array}{cccccc}
\Lambda & I_2 &  & &    \\
 & \Lambda & I_2 &   \\
 &  &  \ddots & \ddots&   \\
 & & & \ddots & I_2   \\
 & & & &  \Lambda   
\end{array} \right) \quad \text{with} \\\  \Lambda = \left(
\begin{array}{cccccc}
a & -b   \\
b & a  
\end{array} \right) \quad \text{and} \\\ I_2 = \left(
\begin{array}{cccccc}
1 & 0   \\
0 & 1  
\end{array} \right),
\end{equation*}
where the complex numbers $a \pm ib, b \neq 0$, are complex conjugated eigenvalues of $D$.

We now state a result about the modulus of continuity for the components of $X$. We believe that by generalizing the methods used in \cite[Theorem 4.2]{LiWangXiao} one can get sharp results of this kind. But an inequality like \eqref{continuity} will suffice for our purposes.

\begin{prop} \label{4mod}
Let the assumptions of Theorem \ref{hausdorff} hold. If the operator $D$ is of the real canonical form, then there exist a continuous modification $X^*$ of $X$,  positive and finite constants $1\leq p_j \leq m$ $(j=1, \ldots, m)$ and $C_{4,1}$, depending only on $D,d$ and $m$ such that for every $j=1, \ldots, m$
\begin{align}\label{modulus}
\limsup_{r \downarrow 0} \sup_{\substack{x,y \in [0,1]^d\\ \tau_E(x-y) \leq r}} \frac{|X^*_j(x)-X^*_j(y)|}{\tau_E(x-y)^{\lambda_j} \big[ \log \big( \frac{1}{\tau_E(x-y)} \big) \big]^{p_j-\frac{1}{2}} } \leq C_{4,1} \quad a.s.
\end{align}
In particular, for every $\varepsilon >0$ $X^*$ satisfies a.s.
\begin{align}\label{continuity}
|X^*_j(x)-X^*_j(y)| \leq C_{4,2} \tau_E(x-y)^{\lambda_j-\varepsilon} \quad \text{for all} \quad x, y \in [0,1]^d
\end{align}
and for every $j=1, \ldots, m$, where $C_{4,2}$ is a positive and finite constant only depending on $D,d$ and $m$.
\end{prop}

\begin{proof}
The proof of this proposition is essentially an extension of the proof in \cite[Proposition 4.1]{MasonXiao}.

Denote the standard basis of $\rmm$ by $(e_1, \ldots, e_m)$. Since $X$ is $(E,D)$-operator-self-similar and has stationary increments, using the polar coordinates with respect to $E$, for all $x,y \in \rd$ we get
\begin{align*}
\mathbb{E} \big[ \big( X_j(x) - X_j(y) \big)^2 \big] & = \mathbb{E} \big[ \langle X(x) - X(y), e_j \rangle^2 \big] \\
& = \mathbb{E} \big[ \langle \tau_E (x-y)^D X\big( l_E(x-y) \big), e_j \rangle^2 \big] \\
& = \mathbb{E} \Big[ \Big( \sum_{k=1}^m X_k\big( l_E(x-y) \big) \langle \tau_E (x-y)^D e_k , e_j \rangle \Big)^2 \Big] .
\end{align*}
From the proof of \cite[Proposition 4.1]{MasonXiao}, we immediately derive
$$\mathbb{E} \big[ \big( X_j(x) - X_j(y) \big)^2 \big] \leq c \, \tau_E (x-y)^{2\lambda_j}  \big|\log  \tau_E (x-y) \big| ^{2(p_j-1)} , $$
where $1 \leq p_j \leq m$ are constants which only depend on $D$ and $0<c<\infty$ is a constant which only depends on $p$ (or $D$) and on $\max_{\theta \in S_E} \mathbb{E} [|X(\theta)|^2]$. Using this, the rest of the proof follows from the proof of \cite[Theorem 4.2]{LiWangXiao} by estimating the tail probabilities for Gaussian random fields in a standard way (see also \cite{Adler, Csaki, Talagrand}).
\end{proof}

We are now able to give a proof of Theorem \ref{hausdorff}.\newline \newline
\noindent \textit{Proof of Theorem \ref{hausdorff}.}
Let $A \in \mathbb{R}^{m \times m}$ be a real invertible matrix such that $\tilde{D} = A^{-1} DA$ is of the real canonical form. Consider the Gaussian random field $Y$ defined by
$$Y(x) = A^{-1} X(x), \quad x \in \rd.$$
Then $Y$ is an $(E,\tilde{D})$-operator-self-similar Gaussian random field in $\rmm$ and has stationary increments. Note that, since $A$ is invertible, the mapping $x \mapsto Ax$ is bi-Lipschitz and, hence, preserves the Hausdorff dimension. So without loss of generality we will assume that $D$ itself is of the real canonical form.

By Proposition \ref{4mod} there exists a continuous modification $X^*$ of $X$. By continuity, $X$ and $X^*$ are indistinguishable. So without loss of generality, we may and will assume that $X$ itself satisfies \eqref{continuity}. Then the upper bounds in \eqref{imdim} and \eqref{graphdim} follow from \eqref{continuity}, Lemma \ref{upperbound} and Remark \ref{remark1}. So it remains to prove the lower bounds in Theorem \ref{hausdorff}. We will do this in a standard way by using Frostman's theorem (see e.g. \cite{Adler, Fal, Kahane, Mattila}). Throughout this proof, let $c$ and $c'$ be positive unspecified constants which might change in every occurence. First we prove the lower bound in \eqref{imdim}. In fact, it suffices to show that for all sufficiently small $\varepsilon>0$ and all
$$0<\gamma <\min \Big\{ m, \frac{ \sum_{k=1}^p \frac{a_k}{1+\varepsilon} \dim W_k + \sum_{i=1}^j (\lambda_j - \lambda_i) }{ \lambda_j }, 1 \leq j \leq m \Big\}$$
the expected energy integral satisfies
\begin{align}\label{1energyintegral}
\mathcal{E}_\gamma = \int_{[0,1]^d \times [0,1]^d} \mathbb{E} [ \|X(x) - X(y) \|^{-\gamma} ] dx dy < \infty .
\end{align}
If $1\leq j \leq m$ is the integer such that $\sum_{i=1}^{j-1} \lambda_i < \sum_{k=1}^p a_k \dim W_k \leq \sum_{i=1}^{j} \lambda_i$ by Lemma \ref{cases} we may assume $j-1 < \gamma < j$. On the other hand, if $\sum_{i=1}^{m} \lambda_i < \sum_{j=1}^{p} a_j \dim W_j$ by Lemma \ref{cases} it suffices to assume that $\gamma <m$. For simplicity, we will only consider the case $\sum_{i=1}^{j-1} \lambda_i < \sum_{k=1}^p a_k \dim W_k \leq \sum_{i=1}^{j} \lambda_i$, since the remaining case is proven analogously. We fix $1\leq j \leq m$ satisfying the latter condition. In order to show \eqref{1energyintegral}, we observe that for any 
$$0<\gamma < \frac{\sum_{k=1}^p \frac{a_k}{1+\varepsilon} \dim W_k + \sum_{i=1}^j (\lambda_j - \lambda_i) }{ \lambda_j }$$
there exists an integer $1 \leq l \leq p$ such that
\begin{small}
\begin{align}\label{imgammainterval}
\frac{ \sum_{k=l+1}^p \frac{a_k}{1+\varepsilon} \dim W_k + \sum_{i=1}^j (\lambda_j - \lambda_i) }{ \lambda_j } \leq \gamma < \frac{ \sum_{k=l}^p \frac{a_k}{1+\varepsilon} \dim W_k + \sum_{i=1}^j (\lambda_j - \lambda_i) }{ \lambda_j } .
\end{align}
\end{small}In the following, we only consider the case $l=1$, since the remaining cases are easier because they require less steps of integration using Lemma \ref{1lbound}. So assuming \eqref{imgammainterval} with $l=1$, we can choose positive constants $\delta_2, \ldots, \delta_p$ such that $\delta_k > \frac{1+\varepsilon}{a_k}$, $2\leq k \leq p$ and
\begin{small}
\begin{align}\label{imgammadeltabound}
 \frac{ \sum_{k=2}^p \frac{\dim W_k}{\delta_k}  + \sum_{i=1}^j (\lambda_j - \lambda_i) }{ \lambda_j } < \gamma 
 < \frac{ \frac{a_1}{1+\varepsilon} \dim W_1 + \sum_{k=2}^p \frac{\dim W_k}{\delta_k}  + \sum_{i=1}^j (\lambda_j - \lambda_i) }{ \lambda_j } .
\end{align}
\end{small}
Note that, since $X$ is $(E,D)$-operator-self-similar with stationary increments
$$ X(x) - X(y) \stackrel{\rm d}{=} \tau_E (x-y) ^D X\big( l_E(x-y) \big)$$
for all $x,y \in \rd$. This implies that
\begin{align*}
\det \cov \big( X(x)-X(y) \big) &= \big( \det \tau_E(x-y)^D \big) ^2 \det \cov X\big( l_E(x-y) \big)\\
& = \prod_{j=1}^m \tau_E (x-y)^{2\lambda_j} \det \cov X\big( l_E(x-y) \big) .
\end{align*}
Since $X$ is Gaussian, continuous and proper, using the fact that $S_E$ is compact and bounded away from $0$ we have
$$ \det \cov X\big( l_E(x-y) \big) \geq \min_{\theta \in S_E} \det \cov X (\theta)>0.$$
For $x \neq y$ let 
$$Y_j (x,y) = \frac{X_j(x)-X_j(y)}{\tau_E (x-y)^{\lambda_j}}, \quad j=1, \ldots, m.$$
Then we have
$$\det \cov \big( Y(x,y) \big) >0.$$
Using this and the fact that $X$ is Gaussian, by exactly the same argument used in \cite[p. 279]{Xiao3} we get
\begin{align*}
\mathbb{E} [ \|X(x) - X(y) \|^{-\gamma} ] & =  \int_{\rmm} \frac{1}{(2 \pi)^{\frac{m}{2}}} \frac{1}{\sqrt{\det \cov \big( Y(x,y) \big) }} \Big[ \sum_{i=1}^m \big( u_i \tau_E (x-y) ^{\lambda_i} \big) ^2 \Big] ^{-\frac{\gamma}{2}} \\
& \quad \quad \times \exp \Big( -\frac{1}{2} u^T \cov \big( Y(x,y) \big) ^{-1} u \Big) du \\
& \leq c \int_{\rmm}  \Big[ \sum_{i=1}^m \big( u_i \tau_E (x-y) ^{\lambda_i} \big) ^2 \Big] ^{-\frac{\gamma}{2}} \exp \Big( -\sum_{i=1}^m u_i^2 \Big) du_1 \ldots du_m,
\end{align*}
In view of the proof of \cite[Theorem 2.1]{Xiao3}, by iterative integration we will obtain
\begin{align*}
\mathbb{E} & [ \|X(x) - X(y) \|^{-\gamma} ] \leq c \tau_E (x-y)^{-\gamma \lambda_j + \sum_{i=1}^j (\lambda_j - \lambda_i)} \\
& \quad \times \int_{\mathbb{R}^{d-j}} \Big[ \sum_{i=j+1}^m \big( u_i \tau_E (x-y) ^{\lambda_i-\lambda_j} \big) ^2  + c\Big] ^{-\frac{\gamma-j}{2}} \exp \Big( -\sum_{i=j+1}^m u_i^2 \Big) du_{j+1} \ldots du_m \\
&  \leq c \tau_E (x-y)^{-\gamma \lambda_j + \sum_{i=1}^j (\lambda_j - \lambda_i)} ,
\end{align*} 
where in the last inequality we used the assumption that $j-1< \gamma <j$.
Hence, by a substitution we get
$$ \mathcal{E}_\gamma \leq c \int_{\| x \| \leq 2} \tau_E (x)^{-\gamma \lambda_j + \sum_{i=1}^j (\lambda_j - \lambda_i)} dx .$$
Let $x = x_1 + \ldots + x_p$ for $x_i \in W_i$, $1 \leq i \leq p$. Since the $W_i$ are orthogonal in the associated Euclidean norm, it follows that $\| x \| \leq 2$ implies $\| x_i \| \leq 2$ for $i =1, \ldots, p$. Then Lemma \ref{pcbounds} yields
$$ \mathcal{E}_\gamma \leq c \int_{\| x_1 \| \leq 2} \ldots \int_{\| x_p \| \leq 2} \big(  \| x_1 \|^{\frac{1+\varepsilon}{a_1}} + \ldots + \| x_p \|^{\frac{1+\varepsilon}{a_p}} \big) ^{-\gamma \lambda_j + \sum_{i=1}^j (\lambda_j - \lambda_i)} dx_1 \ldots dx_p .$$
By using polar coordinates we further get
\begin{align}\label{impolar}
\mathcal{E}_\gamma \leq c \int_{0}^2 dr_1\ldots \int_{0}^2 dr_p \big(  r_1^{\frac{1+\varepsilon}{a_1}} + \ldots + r_p^{\frac{1+\varepsilon}{a_p}} \big) ^{-\gamma \lambda_j + \sum_{i=1}^j (\lambda_j - \lambda_i)} \prod_{j=1}^p r_j^{\dim W_j-1} .
\end{align}
Applying Lemma \ref{1lbound} to \eqref{impolar} with
\begin{equation*}
A= \sum_{j=1}^{p-1} r_j ^{\frac{1+\varepsilon}{a_j}}, \quad p = \gamma \lambda_j - \sum_{i=1}^j (\lambda_j - \lambda_i) \quad \text{and} \quad n = \dim W_p,
\end{equation*}
we integrate out $dr_p$ in the last expression and obtain that
\begin{small}
\begin{align*}
\mathcal{E}_\gamma \leq c' + c \int_{0}^2 dr_1\ldots \int_{0}^2 dr_{p-1} \big(  r_1^{\frac{1+\varepsilon}{a_1}} + \ldots + r_{p-1}^{\frac{1+\varepsilon}{a_{p-1}}} \big) ^{-\gamma \lambda_j + \sum_{i=1}^j (\lambda_j - \lambda_i) + \frac{\dim W_p}{\delta_p}} \prod_{j=1}^{p-1} r_j^{\dim W_j-1} .
\end{align*}
\end{small}
By repeating this procedure for $(p-2)$-times, we derive
\begin{align*}
\mathcal{E}_\gamma & \leq c' + c \int_{0}^2  \big(  r_1^{\frac{1+\varepsilon}{a_1}} \big) ^{-\gamma \lambda_j + \sum_{i=1}^j (\lambda_j - \lambda_i) + \sum_{k=2}^p \frac{\dim W_k}{\delta_k}} r_1^{\dim W_1-1} dr_1 < \infty .
\end{align*}
Note that the last integral is finite since the $\delta_k$ satisfy \eqref{imgammadeltabound}. This proves \eqref{1energyintegral}.

Now we prove the lower bound in \eqref{graphdim}. First consider the case $\dim_{\mathcal{H}} X ([0,1]^d) < m$. Since $\dim_{\mathcal{H}} \Gr X ([0,1]^d) \geq \dim_{\mathcal{H}} X ([0,1]^d)$, \eqref{imdim2} and the corresponding upper bound imply that $\dim_{\mathcal{H}} \Gr X ([0,1]^d) = \dim_{\mathcal{H}} X ([0,1]^d)$ a.s. The dimension of the graph can be larger as the dimension of the range if $\dim_{\mathcal{H}} X ([0,1]^d) = m$. In the latter case, \eqref{imdim2} implies that there exists an integer $1 \leq k \leq p$ such that
\begin{align}\label{lambdabound}
\sum_{j=1}^{k-1} \tilde{a_j} \dim \tilde{W}_j \leq \sum_{i=1}^{m} \lambda_i < \sum_{j=1}^{k} \tilde{a_j} \dim \tilde{W}_j .
\end{align}
By Lemma \ref{cases} it suffices to prove that $\dim_{\mathcal{H}} Gr X ([0,1]^d) \geq \gamma'$ a.s. for all
$$ 0< \gamma' < \sum_{j=1}^{k} \frac{\tilde{a}_j}{\tilde{a}_k} \dim \tilde{W}_j + \sum_{j=k+1}^p \dim \tilde{W}_j + \sum_{i=1}^{m} \big( 1- \frac{\lambda_i}{\tilde{a}_k} \big) .$$
Again by Frostman's theorem it is sufficient to show that
\begin{align}\label{2energyintegral}
\mathcal{G} _\gamma = \int_{[0,1]^d \times [0,1]^d} \mathbb{E} \big[ \big( \|x-y\|^2 + \|X(x) - X(y) \|^2 \big) ^{-\frac{\gamma}{2}} \big] dx dy < \infty
\end{align}
for all
$$ 0< \gamma < \sum_{j=1}^{k} \frac{\tilde{a}_j}{\tilde{a}_k + \varepsilon} \dim \tilde{W}_j + \sum_{j=k+1}^p \dim \tilde{W}_j + \sum_{i=1}^{m} \big( 1- \frac{\lambda_i + \varepsilon}{\tilde{a}_k} \big) ,$$
where $\varepsilon>0$ is arbitrarily small. Furthermore, since $\varepsilon>0$ is arbitrarily small, by Lemma \ref{cases} we may and will assume that
$$\gamma \in \Big( m+ \sum_{j=k+1}^p \dim \tilde{W}_j , m+ \sum_{j=k}^p \dim \tilde{W}_j \Big) .$$
In order to show \eqref{2energyintegral} we use the following well-known fact (see e.g. \cite{Xiao3}) that
\begin{align}\label{densitybound}
\int_0^\infty (s^2 + a^2 )^{-\frac{\gamma}{2}} ds = c(\gamma) a^{-\gamma +1}
\end{align}
for $\gamma >1$ and $c(\gamma)$ is a positive constant which only depends on $\gamma$. Let $Y(x,y)$ be as above. Then
\begin{align*}
\xi & := \mathbb{E} \big[ \big( \|x-y\|^2 + \|X(x) - X(y) \|^2 \big) ^{-\frac{\gamma}{2}} \big] \\
& = \int_{\rmm} \frac{1}{(2 \pi)^{\frac{m}{2}}} \frac{1}{\sqrt{\det \cov (Y)}} \big[ \|x-y\|^2 + \sum_{i=1}^m \big( u_i \tau_E (x-y) ^{\lambda_i} \big) ^2 \big] ^{-\frac{\gamma}{2}} \\
& \quad \quad \times \exp \big( -\frac{1}{2} u \cov (Y)^{-1} u^* \big) du_1 \ldots du_m \\
& \leq c \int_{\rmm} \big[ \|x-y\|^2 + \sum_{i=1}^m \big( u_i \tau_E (x-y) ^{\lambda_i} \big) ^2 \big] ^{-\frac{\gamma}{2}}  du_1 \ldots du_m \\
& = c \, \tau_E(x-y)^{-\gamma \lambda_1} \cdot  \int_{\rmm} \big[ u_1^2 + \frac{\|x-y\|^2}{\tau_E (x-y)^{2\lambda_1}} + \sum_{i=2}^m \big( u_i \tau_E (x-y) ^{\lambda_i-\lambda_1} \big) ^2 \big] ^{-\frac{\gamma}{2}}  du_1 \ldots du_m .
\end{align*}
We first integrate out $du_1$ using \eqref{densitybound} to obtain that
\begin{align*}
\xi \leq c \, \tau_E(x-y)^{-\gamma \lambda_1} \cdot  \int_{\mathbb{R}^{m-1}} \Big[ \frac{\|x-y\|^2}{\tau_E (x-y)^{2\lambda_1}} + \sum_{i=2}^m \big( u_i \tau_E (x-y) ^{\lambda_i-\lambda_1} \big) ^2 \Big] ^{-\frac{(\gamma -1)}{2}} du_2 \ldots du_m .
\end{align*}
Repeating this procedure for $du_2, \ldots, du_m$ we find that
\begin{align*}
\xi & \leq c\, \tau_E(x-y)^{-\gamma \lambda_m + \sum_{i=1}^m (\lambda_m-\lambda_i)} \|x-y\|^{- (\gamma - m)}  \tau_E(x-y)^ {(\gamma - m) \lambda_m } \\
& = c \, \tau_E(x-y)^{-\sum_{i=1}^m \lambda_i} \|x-y\|^{- (\gamma - m)} ,
\end{align*}
so that
\begin{align*}
\mathcal{G} _\gamma &\leq c \int_{[0,1]^d \times [0,1]^d} \tau_E(x-y)^{-\sum_{i=1}^m \lambda_i} \|x-y\|^{- (\gamma - m)} dx dy \\
& \leq c \int_{\|x \| \leq 2} \tau_E(x)^{-\sum_{i=1}^m \lambda_i} \|x\|^{- (\gamma - m)} dx.
\end{align*}
Now we consider two cases. First we consider the case that
$$\sum_{j=1}^{k-1} \tilde{a_j} \dim \tilde{W}_j = \sum_{i=1}^{m} \lambda_i$$
in \eqref{lambdabound}. In this case let us write $x = x_1 + \ldots + x_{k-1} + y$ for $x_i \in \tilde{W}_i$, $1 \leq i \leq k-1$ and $y \in \tilde{W}_k \oplus + \ldots \oplus \tilde{W}_p$. Since the $\tilde{W}_i$ are orthogonal, by the equivalence of norms, we can choose $c>0$ such that
$$\| x \| = \sqrt{\sum_{i=1}^{k-1} \| x_j \|^2 + \|y\|^2} \geq c \big(\sum_{i=1}^{k-1} \| x_j \| + \|y\| \big).$$
Using this and Lemma \ref{pcbounds}, as in the above we obtain
\begin{align*}
\mathcal{G} _\gamma & \leq c \int_{\|x_1 \| \leq 2} \ldots \int_{\|x_{k-1} \| \leq 2} \int_{\|y \| \leq 2} \big( \sum_{j=1}^{k-1} \| x_j \| ^{\frac{1+\varepsilon}{\tilde{a}_j}} + \| y \| ^{\frac{1+\varepsilon}{\tilde{a}_p}} \big) ^{-\sum_{i=1}^{m} \lambda_i} \\
& \quad \quad \times \big( \sum_{j=1}^{k-1} \| x_j \| + \| y \| \big) ^{-(\gamma-m)} dx_1 \ldots dx_{k-1} dy.
\end{align*}
For simplicity of notation let
$$g_j = \frac{\tilde{a}_j}{1+\varepsilon} , \quad  1 \leq j \leq p.$$
By using polar coordinates we get that
\begin{align}\label{graphpolar}
\mathcal{G}_\gamma & \leq c \int_{0}^2 dr \int_{0}^2 dr_{k-1}\ldots \int_{0}^2 dr_1 \big( \sum_{j=1}^{k-1} r_j^{\frac{1}{g_j}} + r^{\frac{1}{g_p}} \big) ^{-\sum_{i=1}^{m} \lambda_i} \\ \nonumber
& \quad \quad \times \big( \sum_{j=1}^{k-1} r_j + r \big) ^{-(\gamma-m)} \prod_{j=1}^{k-1} r_j^{\dim W_j-1} \cdot r^{\dim W_k + \ldots +\dim W_p -1} .
\end{align}
In order to show that the integral in \eqref{graphpolar} is finite, we will integrate $dr_1, \ldots , dr_{k-1}$ iteratively. Furthermore, we will assume that $k>1$ in \eqref{lambdabound} (if $k=1$, we can use \eqref{eql} to obtain \eqref{secondfinite} directly).

We first integrate $dr_1$. Since
$$ \sum_{i=1}^{m} \lambda_i \cdot g_1 > \dim \tilde{W}_1 ,$$
we can use \eqref{eqb} of Lemma \ref{2lbound} with
\begin{equation*}
A = \sum_{j=2}^{k-1} r_j^{\frac{1}{g_j}} + r^{\frac{1}{g_p}} \quad \text{and} \quad B = \sum_{j=2}^{k-1} r_j + r
\end{equation*}
to get that
\begin{align*}
\mathcal{G}_\gamma & \leq c \int_{0}^2 dr \int_{0}^2 dr_{k-1}\ldots \int_{0}^2 dr_2 \big( \sum_{j=2}^{k-1} r_j^{\frac{1}{g_j}} + r^{\frac{1}{g_p}} \big) ^{-\sum_{i=1}^{m} \lambda_i + g_1 \dim \tilde{W}_1} \\ \nonumber
& \quad \quad \times \big( \sum_{j=2}^{k-1} r_j + r \big) ^{-(\gamma-m)} \prod_{j=2}^{k-1} r_j^{\dim W_j-1} \cdot r^{\dim W_k + \ldots +\dim W_p -1} .
\end{align*}
Repeating this procedure for $dr_2, \ldots, dr_{k-2}$ we obtain
\begin{align*}
\mathcal{G}_\gamma & \leq c \int_{0}^2 dr \int_{0}^2 dr_{k-1} \big(  r_{k-1}^{\frac{1}{g_{k-1}}} + r^{\frac{1}{g_p}} \big) ^{-\sum_{i=1}^{m} \lambda_i + \sum_{j=1}^{k-2} g_j \dim \tilde{W}_j} \\ \nonumber
& \quad \quad \times \big( r_{k-1} + r \big) ^{-(\gamma-m)} r_{k-1}^{\dim W_{k-1}-1} \cdot r^{\dim W_k + \ldots +\dim W_p -1} .
\end{align*}
Note that, since we assumed
$$\sum_{j=1}^{k-1} \tilde{a_j} \dim \tilde{W}_j = \sum_{i=1}^{m} \lambda_i,$$
we have to use \eqref{eqn} of Lemma \ref{2lbound} to integrate $dr_{k-1}$ and obtain that
\begin{align}\label{firstfinite}
\mathcal{G}_\gamma & \leq c \int_{0}^2 r ^{-(\gamma-m)} \cdot \log \Big[ 1 + r^{\dim \tilde{W}_{k-1}} \cdot \big( r^{\frac{1}{g_p}} \big)^{- \frac{\dim \tilde{W}_{k-1}}{g_{k-1}}} \Big] \cdot r^{\dim W_k + \ldots +\dim W_p -1}  dr \\ \nonumber
& < \infty .
\end{align}
Note that the last integral is finite since
$$m - \gamma > \sum_{j=k}^p \dim \tilde{W}_j.$$
On the other hand if
$$\sum_{j=1}^{k-1} \tilde{a_j} \dim \tilde{W}_j < \sum_{i=1}^{m} \lambda_i,$$
in \eqref{lambdabound}, we write $x = x_1 + \ldots + x_{k} + y$ for $x_i \in \tilde{W}_i$, $1 \leq i \leq k$ and $y \in \tilde{W}_{k+1} \oplus + \ldots \oplus \tilde{W}_p$. By using the same steps as above we can derive
\begin{align*}
\mathcal{G}_\gamma & \leq c \int_{0}^2 dr \int_{0}^2 dr_{k} \int_{0}^2 dr_{k-1} \big(  r_{k-1}^{\frac{1}{g_{k-1}}} + r_k^{\frac{1}{g_k}} + r^{\frac{1}{g_{p}}} \big) ^{-\sum_{i=1}^{m} \lambda_i + \sum_{j=1}^{k-2} g_j \dim \tilde{W}_j} \\ \nonumber
& \quad \quad \times \big( r_{k-1} + r_k + r \big) ^{-(\gamma-m)} r_{k-1}^{\dim W_{k-1}-1} \cdot r_{k}^{\dim W_{k}-1} \cdot r^{\dim W_{k+1} + \ldots +\dim W_p -1} .
\end{align*}
Note that we have to use \eqref{eqb} again in order to integrate $dr_{k-1}$ and this gives
\begin{align}\label{almostdone}
\mathcal{G}_\gamma & \leq c \int_{0}^2 dr \int_{0}^2 dr_{k} \big( r_k^{\frac{1}{g_k}} + r^{\frac{1}{g_{p}}} \big) ^{-\sum_{i=1}^{m} \lambda_i + \sum_{j=1}^{k-1} g_j \dim \tilde{W}_j} \\ \nonumber
& \quad \quad \times \big( r_k + r \big) ^{-(\gamma-m)} r_{k}^{\dim W_{k}-1} \cdot r^{\dim W_{k+1} + \ldots +\dim W_p -1} .
\end{align}
We now integrate $dr_k$ in \eqref{almostdone} by using \eqref{eql} and we see that
\begin{align}\label{secondfinite}
\begin{split}
\mathcal{G}_\gamma & \leq c \int_{0}^2 r  ^{m - \gamma - \frac{1}{g_k} \big( \sum_{i=1}^{m} \lambda_i - \sum_{j=1}^{k-1} g_j \dim \tilde{W}_j \big) + \dim \tilde{W}_k} \cdot r^{\dim W_{k+1} + \ldots +\dim W_p -1} dr +c\\ 
& < \infty,
\end{split}
\end{align}
since
\begin{align*}
m - \gamma - \frac{1}{g_k} \Big( \sum_{i=1}^{m} \lambda_i - \sum_{j=1}^{k-1} g_j \dim \tilde{W}_j \Big) + \dim \tilde{W}_k > \sum_{j=k+1}^p \dim \tilde{W}_j
\end{align*}
for arbitrarily small $\varepsilon>0$. Combining \eqref{firstfinite} and \eqref{secondfinite} yields \eqref{2energyintegral}. This completes the proof of Theorem \ref{hausdorff}. \hfill $\Box$
\newline \newline
\appendix
\section{} \label{}
Denote
\begin{align*}
	\zeta &= \min \{ m, \frac{ \sum_{k=1}^p a_k \dim W_k + \sum_{i=1}^j (\lambda_j - \lambda_i) }{ \lambda_j }, 1 \leq j \leq m \} \\
\end{align*}
and
\begin{align*}
	\kappa &= \min \{\frac{ \sum_{k=1}^p a_k \dim W_k + \sum_{i=1}^j (\lambda_j - \lambda_i) }{ \lambda_j }, 1 \leq j \leq m, \\ 
	& \quad \sum_{j=1}^l \frac{\tilde{a}_j}{\tilde{a}_l} \dim \tilde{W}_j + \sum_{j=l+1}^p \dim \tilde{W}_j + \sum_{i=1}^m (1- \frac{\lambda_i}{\tilde{a}_l}) , 1 \leq l \leq p \}.
\end{align*}

\begin{lemma}\label{cases}
The following statements hold.
\begin{enumerate}[label=(\roman*)]
\item Assume that $\sum_{i=1}^{l-1} \lambda_i < \sum_{k=1}^p a_k \dim W_k \leq \sum_{i=1}^{l} \lambda_i$ for some $1 \leq l \leq m$. Then
$$\zeta = \frac{ \sum_{k=1}^p a_k \dim W_k + \sum_{i=1}^l (\lambda_l - \lambda_i) }{ \lambda_l }$$
and $\zeta \in (l-1,l]$. 
\item If there is $1 \leq k \leq p$ such that $\sum_{j=1}^{k-1} \tilde{a}_j \dim \tilde{W}_j \leq \sum_{i=1}^{m} \lambda_i < \sum_{j=1}^{k} \tilde{a}_j \dim \tilde{W}_j$ then $\zeta = m$,
$$ \kappa = \sum_{j=1}^k \frac{\tilde{a}_j}{\tilde{a}_k} \dim \tilde{W}_j + \sum_{j=k+1}^p \dim \tilde{W}_j + \sum_{i=1}^m (1- \frac{\lambda_i}{\tilde{a}_k})$$
and $\kappa \in \big( m + \sum_{j=k+1}^p \dim \tilde{W}_j, m + \sum_{j=k}^p \dim \tilde{W}_j\big]$.
\end{enumerate}
\end{lemma}

In order to prove Theorem \ref{hausdorff} we state and prove two lemmas which are needed in establishing the lower bounds in \eqref{imdim} and \eqref{graphdim}. Lemma \ref{1lbound} below with $n=1$ is an analogous statement to \cite[Lemma 3.6]{AychXiao} (see also \cite[p.212]{XiaoZhang}). Further, Lemma \ref{2lbound} with $n=1$ is the statement of \cite[Lemma 3.7]{AychXiao}. By using the methods in \cite{XiaoZhang, AychXiao} we can establish the statements for general $n \in \mathbb{N}$.

\begin{lemma}\label{1lbound}
Let $0<h<1$ be a given constant. Then, for any constants $\delta >h, M>0$, $p>0$ and any $n \in \mathbb{N}$, there exist positive and finite constants $C_{4,3}$ and $C_{4,4}$, depending only on $\delta, p$ and $M$ such that for all $0<A \leq M$
\begin{align} \label{IA}
I(A)  := \int_0^2 (A+r^h)^{-p} r^{n-1} dr \leq C_{4,3} (A^{-p+\frac{n}{\delta}}+C_{4,4} ).
\end{align}
\end{lemma}
\begin{proof}
For the sake of completeness let us give a proof. Throughout this proof, let $c$ be an unspecified positive constant which might change in each occurence. We first assume that $p = \frac{n}{\delta}$. In this case we have
\begin{align*}
I(A)  &= \int_0^2 (A+r^h)^{-p} r^{n-1} dr \leq \int_0^2 r^{-ph} r^{n-1} dr = \int_0^2 r^{-\frac{n}{\delta}h+n-1} dr < \infty,
\end{align*}
since $\delta >h$ by assumption. So it suffices to prove \eqref{IA} for $p \neq \frac{n}{\delta}$. By using the substitution $s = A+r^h$ we get
\begin{align*}
I(A)  &= \int_0^2 (A+r^h)^{-p} r^{n-1} dr = c \int_A^{A+2^h} s^{-p} (s-A)^{\frac{n-h}{h}} ds .
\end{align*}
Since $0<h<1$, elementary calculation shows that for all $0<A\leq M$
\begin{align*}
I(A)  & \leq c \int_A^{A+2^h} s^{-p+\frac{n-h}{h}} ds \leq c(M+2)^{\frac{n}{h}-\frac{n}{\delta}} \int_A^{A+2^h} s^{-p-1 + \frac{n}{\delta}} ds \\
& \leq C_{4,3} (A^{-p+\frac{n}{\delta}}+C_{4,4} )
\end{align*}
and the proof is complete.
\end{proof}
\begin{lemma}\label{2lbound}
Let $\alpha$, $\beta$, $\eta$ be positive constants and $n \in \mathbb{N}$. For $A>0$ and $B>0$ define
\begin{align}
J := J(A,B) = \int_0^2 \frac{r^{n-1}}{(A+r^\alpha)^\beta (B+r)^\eta} dr.
\end{align}
Then there exist positive constants $C_{4,5}$ and $C_{4,6}$, depending only on $\alpha$, $\beta$, $\eta$ such that the following holds for all reals $A,B>0$ satisfying $A^{\frac{1}{\alpha}} \leq C_{4,5}B$:
\begin{enumerate}[label=(\roman*)]
\item if $\alpha\beta >n$, then
\begin{align}\label{eqb}
J & \leq C_{4,6} \frac{1}{A^{\beta-\frac{n}{\alpha}}B^\eta}
\end{align}
\item if $\alpha\beta =n$, then
\begin{align}\label{eqn}
J & \leq C_{4,6} \frac{1}{B^\eta} \log (1+B^n A^{-\frac{n}{\alpha}} )
\end{align}
\item if $0<\alpha\beta <n$ and $\alpha \beta + \eta \neq n$, then
\begin{align}\label{eql}
J & \leq C_{4,6} (\frac{1}{B^{\alpha \beta + \eta - n}} + 1) .
\end{align}
\end{enumerate}
\end{lemma}
\begin{proof}
Since the proof of this lemma is essentially identical to the proof that can be found in \cite[Lemma 3.7]{AychXiao} and is carried out by using exactly the same steps, we omit it and leave it to the reader.
\end{proof} 

\section*{Acknowledgements}
The author would like to thank the anonymous referees for valuable comments which helped to improve the manuscript.

This work has been supported by Deutsche Forschungsgemeinschaft (DFG) under grant KE1741/ 6-1

\bibliographystyle{plain}

\end{document}